\documentclass[12pt]{article}

\usepackage{amsmath, amsfonts,amssymb,  amscd, amsthm, verbatim, mathtools} %, amsxtra}
\usepackage{fullpage}
\usepackage[all]{xy}

\DeclareMathOperator{\Spec}{Spec}

\def\L{\mathcal{L}}
\def\A{\mathbb{A}}
\def\Z{\mathbb{Z}}

\def\G{\mathbb{G}}

\def\P{\mathbb{P}}
\def\O{\mathcal{O}}
\def\longto{\longrightarrow}
\def\onto{\twoheadrightarrow}
\def\into{\hookrightarrow}

\def\F{\mathbb{F}}

\theoremstyle{plain}

\newtheorem{thm}{Theorem}

\newtheorem{Lem}[thm]{Lemma}

\newtheorem{Prop}[thm]{Proposition}
\newtheorem{Conj}[thm]{Conjecture}
\newtheorem{Def}[thm]{Definition}

\theoremstyle{remark}
\newtheorem{Rem}[thm]{Remark}
\newtheorem{Que}[thm]{Question}
\newtheorem{Exa}[thm]{Example}

\begin{document}

\title{An algebraic geometry version of the Kakeya problem}
\author{Kaloyan Slavov}
\maketitle
\begin{abstract}
We propose an algebraic geometry framework for the Kakeya problem. We conjecture that for any polynomials $f,g\in\F_{q_0}[x,y]$ and any $\F_q/\F_{q_0}$, the image of the map $\F_q^3\to\F_q^3$ given by $(s,x,y)\mapsto (s,sx+f(x,y),sy+g(x,y))$ has size at least $\frac{q^3}{4}-O(q^{5/2})$ and prove the special case when $f=f(x), g=g(y).$ We also prove it in the case $f=f(y), g=g(x)$ under the additional assumption $f'(0)g'(0)\neq 0$ when $f,g$ are both linearized.  Our approach is based on a combination of Cauchy--Schwarz and Lang--Weil.
The algebraic geometry inputs in the proof  are various results concerning irreducibility of certain classes of multivariate polynomials. 
\end{abstract}


\section{Introduction}

The Kakeya problem is a major open problem in classical harmonic analysis: if a compact subset $E\subset\mathbb{R}^n$ contains a unit line segment in every direction,  then $E$ has Hausdorff and Minkowski dimension $n.$ This is known for $n=2$; see \cite{Tao_blog_Kakeya} for a survey, history, and references. In 1999, T. Wolff \cite{W} proposed a finite field model for the Kakeya problem:
if $E\subset\F_q^n$ contains a line in any direction, then $|E|\geq c_n q^n,$ for some $c_n$ which depends only on $n$.  The finite field Kakeya problem has proved to be a 
useful model for the classical much harder Euclidean problem. 
After a long period of frustration, the finite field problem was proved by Z. Dvir in \cite{Dvir} by a short and elegant argument
based on the polynomial method. In brief, if $E\subset\F_q^n$ is a Kakeya subset of small size, one can find a hypersurface $V(f)$ over $\F_q$ of degree $d<q$ which vanishes on $E$. Then the condition that $E$ is Kakeya will force the homogeneous piece of $f$ of top degree
to vanish on all of $\P^{n-1}(\F_q),$ and this contradicts the Schwartz--Zippel lemma.

We propose an algebraic geometry version of the Kakeya problem. The main motivation is that the smallest known example of a Kakeya subset of $\F_q^n$ comes from
\[\{(a_1,...,a_{n-1},b)\in\F_q^n\ |\ a_i+b^2\ \text{is a square in $\F_q$\ for all $i$}\}\subset\F_q^n\]
(say $q$ is odd for convenience; see \cite{SarafSudan}). 
Our starting observation is that this is in fact the image on $\F_q$-points of
\[\xymatrix{
V\left(a_1+b^2-c_1^2,...,a_{n-1}+b^2-c_{n-1}^2\right)
\ar@{^{(}->}[r]\ar[rd] & \A^{2n-1}_{a_1,...,a_{n-1},b,c_1,...,c_{n-1}}\ar[d]\\
{}& \A^n_{a_1,...,a_{n-1},b}
}\]
So, this Kakeya {\it subset} of $\F_q^n$ comes from a {\it variety} already defined over $\F_p$ (in fact, over $\Z$) and hence inherits extra structure, which should not be neglected. We give a definition of a ``Kakeya variety" that models this example.  

We define a Kakeya variety over a base field, generalizing the example coming from the quadric hypersurfaces. In brief, let $E$ be a variety over a base field $k_0$, together with a morphism $E\to\P^n_{k_0}$ over $k_0$. Let $H_0=V(x_0)$ be the hyperplane at infinity, and so $H_0\simeq \P^{n-1}$ parametrizes the directions of lines in $\P^n$ not contained in $H_0$. There is a variety $F(E)$ over $k_0$ such that for a field $K/k_0$, the set $F(E)(K)$ consists of all $K$-morphisms $\P^1_K\to E_K$ such that the composition $\P^1_K\to E_K\to\P^n_K$ gives rise to a line not contained in $H_0$. We say that $(E,E\to\P^n)$ is Kakeya if the direction map $F(E)\to H_0$ has a rational section. 

A Kakeya variety in this strong algebraic sense over a finite field $\F_{q_0}$ gives rise to a Kakeya subset $E_{\F_q}$ of $\F_q^n$ (after adding $O(q^{n-1})$ points if necessary), for any $\F_q/\F_{q_0}$, by taking image on $\F_q$-points in the affine chart. Our goal here is to give a lower bound for 
$\#E_{\F_q}$ by using a uniform geometric argument, which, ideally, refers only to the base field $\F_{q_0}$ and its algebraic closure $\overline{\F_p}$. Note that Dvir's proof uses a hypersurface of degree $d<q$ for a Kakeya subset of $\F_q^n$, hence it is specific to the given $\F_q^n$. In other words,
for {\it each} $\F_q/\F_{q_0}$, Dvir's argument for the size of $E_{\F_q}$ would pick a different hypersurface, whose degree varies with $q$. Our project, however, is to give a uniform geometric argument for all $\F_q/\F_{q_0}$ at once. Such an argument would give further understanding of the geometry behind the Kakeya problem. 

We emphasize that our goal is not to redo the finite field Kakeya problem, which is already known anyways. Rather, our goal is to give an algebraic geometry {\it framework} for the Kakeya problem. Our investigation leads to interesting algebraic geometry questions on their own right (specifically, questions about reducibility of certain classes of polynomials), and we hope that, conversely, our  approach  might interact with previous classical frameworks for the Kakeya problem. For any (combinatorial) Kakeya subset $E_0\subset\F_q^n$, we can find a Kakeya variety $E$ over $\F_q$ such that $E_0$ arises from the $\F_q$-points of $E$; however, $E$ may have large complexity, and since the error terms in our approach depend on the complexity of $E$, this will not be useful for a bound on the size of the specific $E_0$ (again, this is not our goal). The algebraic geometry tools that we use are suitable for the regime when $q$ becomes large relative to the complexity of $E\to\P^n$. 

Specifically, let $n=3$ and consider a Kakeya variety $E\to\P^3$ over $\F_{q_0}$. Let $E_{\F_q}$ be the image on $\F_q$-points. We conjecture that
\[|E_{\F_q}|\geq\frac{q^3}{4}-O(q^\frac{5}{2})\]
(where the implied constant depends on the complexity of $E\to\P^n$).
Making explicit the algebraic Kakeya condition, this statement is essentially the following:

\begin{Conj}
Let $L(t_1,t_2),M(t_1,t_2)\in\F_{q_0}[t_1,t_2]$ be arbitrary polynomials
in two variables.
Consider the map
\begin{align*}
\varphi:\A^3_{\F_{q_0}} &\longrightarrow \A^3_{\F_{q_0}}\\
(s,t_1,t_2) &\longmapsto (s,st_1+L(t_1,t_2),st_2+M(t_1,t_2)).
\end{align*}
For each extension $\F_q/\F_{q_0}$, let $E_{\F_q}$ be the image of the induced map $\A^3(\F_q)\to\A^3(\F_q)$ on 
$\F_q$-points. Then
\[|E_{\F_q}|\geq\frac{q^3}{4}-O(q^\frac{5}{2}),\]
where the implied constant depends only on the degrees of $L$ and $M$.
\label{main_conj}
\end{Conj}

We prove the following extreme special case\footnote{Under a technical assumption $p\geq 5$ on the characteristic.}:
\begin{Prop} Assume that $L(t_1,t_2)=L(t_1)$ 
and $M(t_1,t_2)=M(t_2)$ depend only on the first or second variable, respectively. Then Conjecture \ref{main_conj} holds true. 
\label{separated_vars}
\end{Prop}

A polynomial $f(x)\in\overline{\F_{p}}[x]$ is called linearized if it is of the form $f(x)=\sum a_i x^{p^i}+f(0).$ We also prove 
\begin{Prop}
Assume that $L(t_1,t_2)=L(t_2), M(t_1,t_2)=M(t_1)$ are polynomials over $\F_{q_0}$. If $L$ and $M$ are linearized polynomials, assume in addition that
$L'(0)M'(0)\neq 0.$ 
Then Conjecture \ref{main_conj} holds true. 
\label{mixed_vars_prop}
\end{Prop}

A bound 
with error term of this form is what we may hope for, using geometric tools. 
It is reasonable to think that the special cases that we have resolved are in fact ``the worst" cases for the conjecture, hence provide sufficient evidence. 
We remark that the smallest known Kakeya subset of $\F_q^3$ has size of order $\frac{q^3}{4},$ and the best known lower bound
is for order of $\frac{q^3}{8}.$ Thus, our approach and conjecture would give some evidence that indeed, $\frac{q^3}{4}$ is the order of the smallest Kakeya subset of $\F_q^3$.  

Our method is based on the Cauchy--Schwarz inequality and the Lang--Weil bound, and is inspired by the following easy combinatorial proof of the $2$-dimensional finite field Kakeya problem, known as Davies's approach. Namely, let $E\subset\F_q^2$ be a Kakeya subset. Pick lines $L_1,...,L_{q+1}$ contained in $E$, one in each direction. 
Let $I=\{(p,i)\ |\ p\in L_i\}$. Consider the fiber product diagram

\[
\xymatrix{
  & I\times_E I\makebox[0pt][l]{${}=\{(p,i,j)\mid p\in L_i, p\in L_j\}$} \ar[dl] \ar[dr] \\
  I\ar[rd] && I \ar[dl]\\
  & E
}\hspace{15em}% adjust to suit
\]
A lower bound for $I\times_E I$ is given by the Cauchy--Schwarz inequality, and an upper bound follows by splitting the cases $i=j$ (diagonal) and $i\neq j$. Neglecting error terms of smaller order,
\[\frac{q^4}{|E|}=
\frac{|I|^2}{|E|}\leq |I\times_E I|\leq q^2+q^2\quad
\Longrightarrow \quad |E|\geq\frac{q^2}{2}.\]

We give an algebraic geometry version of this argument. 
It is interesting to note that it is this combinatorial proof (rather than Dvir's polynomial method) that interacts best with our algebraic geometry Kakeya problem.

\section{Definition of a Kakeya variety}
\label{maindef_section}

\subsection{Some technical preparations}

Fix a base field $k_0$, a variety $E$ over $k_0,$ and a morphism $E\to\P^n_{k_0}$ defined over $k_0.$ In this discussion, variety over $k_0$ means just a scheme of finite type over $k_0$. 

By Theorem 5.23 in \cite{F},
there exists a scheme $\mathfrak{M}or_{k_{0}}(\P^1_{k_{0}},E)$ such that for any variety $T$ over $k_{0},$ the set
$\mathfrak{M}or_{k_{0}}(\P^1_{k_{0}},E)(T)$ consists of all $T$-morphisms 
$\P^1_{T}\to E_T$, where $E_T=E\times_{k_{0}}T$. Similarly, let 
$\mathfrak{M}or_{k_{0}}(\P^1_{k_{0}},\P^n_{k_{0}})$ be the scheme whose $T$-points, for a scheme $T/k_{0}$, 
are the $T$-morphisms $\P^1_T\to \P^n_T$. Note that 
the given morphism $E\to\P^n_{k_{0}}$ induces
$\mathfrak{M}or_{k_{0}}(\P^1_{k_{0}},E)\to \mathfrak{M}or_{k_{0}}(\P^1_{k_{0}},\P^n_{k_{0}})$.

Next, we define a scheme $\text{Lin}_{k_0}(\P^1_{k_0},\P^n_{k_0})$ which parametrizes morphisms $\P^1\to\P^n$ whose images are lines, as 
\[\text{Lin}_{k_0}(\P^1_{k_0},\P^n_{k_0})=\bigcup_{i\neq j}D_+(z_iy_j-z_jy_i)\subset\P^{2n+1}_{[z_0:y_0:...:z_n:y_n]},\]
with the induced open subscheme structure
(for a homogeneous $f\in k_0[z_0,y_0,...,z_n,y_n]$, we denote by $D_+(f)$ the locus of invertibility of $f$). Note that $\text{Lin}_{k_0}(\P^1_{k_0},\P^n_{k_0})$ is a variety over $k_0.$

Before we state the Lemma below, note that if $K$ is a field, and $K[x_0,...,x_n]\onto K[u,v], x_i\mapsto \alpha_i u+\beta_i v$
is a surjection of $K$-algebras, then the induced map $\P^1_K\into\P^n_K$ gives rise to a line if and only if for some $i\neq j,$ we have $\alpha_i\beta_j-\alpha_j\beta_i\neq 0$. 

\begin{Lem}
There is a morphism 
\[\text{Lin}_{k_0}(\P^1_{k_0},\P^n_{k_0})\to \mathfrak{M}or_{k_{0}}(\P^1_{k_{0}},\P^n_{k_{0}})\]
over $k_0$ such that for any field $K/k_0,$ the induced map on $K$-points sends 
$[\alpha_0:\beta_0:\dots:\alpha_n:\beta_n]\in\text{Lin}_{k_0}(\P^1_{k_0},\P^n_{k_0})(K)$ to the $K$-morphism
$\P^1_K\to\P^n_K$ given by $[u:v]\mapsto [...:\alpha_i u+\beta_i v:\dots].$

In particular, a $K$-morphism $\P^1_K\to\P^n_K$, regarded as an element in
$\mathfrak{M}or_{k_{0}}(\P^1_{k_{0}},\P^n_{k_{0}})(K)$,
 determines a line if and only if it comes from
$\text{Lin}_{k_0}(\P^1_{k_0},\P^n_{k_0})(K)$. 
\end{Lem}

\begin{proof}
 It suffices to describe this map on $S$-points, where $S=\Spec R$ is affine. 
Let $(\L,\L\into\O_S^{2n+2})$ be a point in the set $\text{Lin}_{k_0}(\P^1_{k_0},\P^n_{k_0})(S)\subset\P^{2n+1}(S),$
where $\L$ is a line bundle on $S$, and $\L\into\O_S^{2n+2}$ has locally free cokernel. We have to describe how it gives rise to a morphism $\P^1_S\to\P^n_S.$ Take an affine open cover $S=\cup S_i$ such that $\L_{S_i}$ is trivial for each $i$; it suffices to describe the maps $\P^1_{S_i}\to\P^n_{S_i}$ for each $i$, and hence, replacing $S$ by $S_i$, we can assume that $\L\simeq\O_S$ is trivial on $S$. Thus, we are given
\begin{align*}
R &\hookrightarrow R^{2n+2}\\
1 &\mapsto (\alpha_0,\beta_0,...,\alpha_n,\beta_n)
\end{align*}  
such that $\alpha_0,...,\beta_n$ generate the unit ideal in $R$, and the condition that $S\to\P^{2n+1}$ factors
through $\text{Lin}_{k_0}(\P^1_{k_0},\P^n_{k_0})$ means that the ideal in $R$ generated by 
$\alpha_i\beta_j-\alpha_j\beta_i$ is the unit ideal. 

We claim that in this setting, the $R$-algebra map
\begin{align*}
R[x_0,...,x_n] &\to R[u,v]\\
x_i &\mapsto \alpha_i u+\beta_i v
\end{align*}
is surjective, hence induces a morphism $\P^1_R\to\P^n_R$. Say $r_{ij}\in R$ (for each $i<j$) are such that
$\sum_{i<j}r_{ij}(\alpha_i\beta_j-\alpha_j\beta_i)=1$. For each $i<j$, note that
\[r_{ij}(\alpha_i\beta_j-\alpha_j\beta_i)u=r_{ij}\beta_j(\alpha_i u+\beta_i v)-r_{ij}\beta_i(\alpha_j u+\beta_j v)\]
belongs to the image of the map above; summing over all $i<j$ shows that $u$ belongs to the image, and similarly for $v$. 

The description of the map on $K$-points follows directly from the construction.
\end{proof}

Let $H_0=V(x_0)\subset\P^n_{k_0}$, and consider also 
$\text{Lin}_{k_0}(\P^1_{k_0},H_0):=\text{Lin}_{k_0}(\P^1_{k_0},\P^n_{k_0})\cap V(z_0,y_0)$;
this scheme parametrizes now morphisms $\P^1\to\P^n$ which give rise to lines contained in the
hyperplane $V(x_0)$. Define
\[\text{Lin}^0_{k_0}(\P^1_{k_0},\P^n_{k_0}):=\text{Lin}_{k_0}(\P^1_{k_0},\P^n_{k_0})-
\text{Lin}_{k_0}(\P^1_{k_0},H_0).\]
This scheme parametrizes morphisms $\P^1\to\P^n$ which give rise to lines not contained in $H_0.$

Next, there is a morphism $\text{Lin}_{k_0}^0(\P^1_{k_0},\P^n_{k_0})\to V(x_0)$ which takes a line
not contained in $V(x_0)$
and sends it to its intersection with the hyperplane $V(x_0)$. More formally, 

\begin{Lem}
There is a morphism
\[\text{Lin}_{k_0}^0(\P^1_{k_0},\P^n_{k_0})\to V(x_0)\]
over $k_0$ such that for any field $K/k_0,$ the induced map 
$\text{Lin}_{k_0}^0(\P^1_{k_0},\P^n_{k_0})(K)\to V(x_0)(K)$ is described as follows: a 
$K$-morphism $\P^1_K\to\P^n_K$ is sent to the unique point in the image of $\P^1_K(K)\to\P^n_K(K)$ which belongs to $V(x_0)(K)$. 
\end{Lem}

\begin{proof}
Let $S=\Spec R$ be an affine scheme over $k_0$. We have to describe the map of sets
$\text{Lin}_{k_0}^0(\P^1_{k_0},\P^n_{k_0})(S)\to V(x_0)(S).$ Let $(\L,\L\into\O_S^{2n+2})$ be an element of  $\text{Lin}_{k_0}^0(\P^1_{k_0},\P^n_{k_0})(S)$, where $\L$ is a line bundle on $S$, and
$\L\into\O_S^{2n+2}$ is an injection with a locally free cokernel. We have to associate to it a morphism $S\to V(x_0).$ Take an affine open cover $S=\cup S_i$ with $\L_{S_i}\simeq\O_{S_i}$;
it suffices to describe the maps $S_i\to V(x_0)$. Replacing $S$ by $S_i,$ we can assume that
$\L\simeq\O_S$ is trivial. 

So, we are given an injection of $R$-modules 
$R\into R^{2n+2}, 1\mapsto (\alpha_0,\beta_0,...,\alpha_n,\beta_n)$ with a locally free cokernel.
 We know that this map $\Spec R\to\P^{2n+1}_{k_0}$ factors through 
\[\text{Lin}_{k_0}^0(\P^1_{k_0},\P^n_{k_0})=\text{Lin}_{k_0}(\P^1_{k_0},\P^n_{k_0})\cap
(D_+(z_0)\cup D_+(y_0)).\]
This means that the ideals $I_1=\langle \alpha_i\beta_j-\alpha_j\beta_i\ |\ i\neq j\rangle$ and
$I_2=\langle \alpha_0,\beta_0\rangle$ of $R$ are both equal to the unit ideal $R$.

For $i=1,...,n,$ define $x_i=-\alpha_i\beta_0+\beta_i\alpha_0\in R$. We claim that the ideal
$I=\langle x_i\rangle\subset R$ is the unit ideal. Note that for any $i\neq j,$
$\alpha_0(\alpha_j\beta_i-\alpha_i\beta_j)=\alpha_j x_i-\alpha_i x_j\in I$ and similarly 
$\beta_0(\alpha_j\beta_i-\alpha_i\beta_j)\in I$. Thus, $R=I_1I_2\subset I\subset R$ and hence $I=R$. 

Therefore, the $R$-module map $R\into R^{n+1}, 1\mapsto (0,x_1,...,x_n)$ is injective on all residue fields of $R$, hence gives rise to a morphism $\Spec R\to V(x_0)\into\P^n_{k_0}.$

When $S=\Spec K$ with $K$ a field, the description of the map in the statement of the Lemma follows
from the construction. 
\end{proof}

\subsection{Kakeya variety over a base field}

We now go back to the morphism $E\to\P^n_{k_0}$. Define $F(E)$ as the fiber product in the following diagram:
\[
\xymatrix{
\mathfrak{M}or_{k_{0}}(\P^1_{k_{0}},E) \ar[r]&  \mathfrak{M}or_{k_{0}}(\P^1_{k_{0}},\P^n_{k_0})\\
F(E)\ar[u]\ar[r] & \text{Lin}^0_{k_0}(\P^1_{k_0},\P^n_{k_0})   \ar[u]}
\]
In particular, $F(E)$ is a variety over $k_0$, and for a field $K/k_0,$ the set $F(E)(K)$ consists of all $K$-morphisms $\P^1_K\to E_K$ such that the composition $\P^1_K\to E_K\to \P^n_K$ gives rise to a {\it line} in $\P^n_K$ which is not contained in $V(x_0)$.

Let $k_{0}$ be any field. Consider a variety $E$ over $k_{0},$ together with a morphism
$E\to\P^n_{k_{0}}$ of varieties over $k_{0}.$ 
Take coordinates $[x_0:...:x_n]$ on $\P^n_{k_0},$ and consider the hyperplane $H_0=V(x_0)$. For an open $U\subset V(x_0)$, let $F(E,U)$ 
be the preimage of $U$ in $F(E)$ under 
$F(E)\to \text{Lin}^0_{k_0}(\P^1_{k_0},\P^n_{k_0})\to V(x_0)$.

\begin{Def}
We say that $(E,E\to\P^n_{k_0})$ is a Kakeya variety over $k_0$ if there exists a nonempty open $U\subset\P^n_{k_0}$ such that the morphism $F(E,U)\to U$ has a section. 
\end{Def}
%
%\[
%\xymatrix{
%\mathfrak{M}or_{k_{0}}(\P^1_{k_{0}},E) \ar[r]&  \mathfrak{M}or_{k_{0}}(\P^1_{k_{0}},\P^n_{k_0}) &{}\\
%F(E)\ar[r]\ar[u] & \text{Lin}^0_{k_0}(\P^1_{k_0},\P^n_{k_0}) \ar[r] \ar[u]& V(x_0)\\
%F(E,U)
%\ar@{^{(}->}[u]
%\ar[rr] &{}&U\ar@{^{(}->}[u]\ar@{.>}@/_1pc/[ll]
%}
%\]

\[
\xymatrix{
\mathfrak{M}or_{k_{0}}(\P^1_{k_{0}},E) \ar[r]&  \mathfrak{M}or_{k_{0}}(\P^1_{k_{0}},\P^n_{k_0}) &{}\\
F(E)\ar[r]\ar[u] & \text{Lin}^0_{k_0}(\P^1_{k_0},\P^n_{k_0}) \ar[r] \ar[u]& V(x_0)\\
F(E,U)
\ar@{^{(}->}[u]
\ar[rr] &{}&U\ar@{^{(}->}[u]\ar@{.>}@/_1pc/[ll]
}
\]

\begin{Rem}
If $k_0=\F_{q_0}$ is a finite field and $\dim F(E)=n-1$,
 we may instead impose the requirement that 
 for some open $U\subset V(x_0)$, the morphism
$F(E,U)\to U$ is separable, and
for some irreducible component $Z$ of $F(E)$, the map $Z(K)\to U(K)$ is surjective, for any finite field $K/\F_{q_0}$.  It is known that this implies that $F(E,U)\to U$ is birational, hence $E$ will be Kakeya. 
\end{Rem}

\begin{Exa}
Let $k_0$ be any field (suppose $\text{char}k_0\neq 2$ for convenience; a small modification is needed in characteristic $2$).
 Let $E=V(a_1x_0+b^2-c_1^2,...,a_{n-1}x_0+b^2-c_{n-1}^2)\subset\P^{2n-1}_{[x_0:a_1:...:a_{n-1}:b:c_1:...:c_{n-1}]}$ and consider the map $E\to\P^n_{[x_0:a_1:...:a_{n-1}:b]}$
induced by projection onto the first $n+1$ coordinates. Take $U=D_+(b)\subset V(x_0)\subset\P^n_{[x_0:a_1:...:a_{n-1}:b]}$, with $U\simeq \A^{n-1}_{\alpha_1,...,\alpha_{n-1}}$. For $S=\Spec R$,
the map $U(S)\to \{\text{$S$-morphisms}\ \P^1_S\to E_S\}$ is described as follows. An element $(\alpha_1,...,\alpha_{n-1})\in R^{n-1}$ induces a surjection
\begin{align*}
R[x_0,a_1,...,a_{n-1},b,c_1,...,c_{n-1}]/\langle a_i w+b^2-c_i^2\rangle &\longto R[t,t_1]\\
x_0 &\longmapsto t_1\\
a_i&\longmapsto \alpha_i t+\frac{\alpha_i^2}{4}t_1\\
b &\longmapsto t\\
c_i &\longmapsto t+\frac{\alpha_i}{2}t_1
\end{align*}
of $R$-algebras, which in turn gives rise to $\P^1_R\to E_R$.

The smallest known example of a Kakeya subset of $\F_q^n$ arises from this Kakeya variety when $k_0=\F_p$.
\end{Exa}

\begin{Exa}
If we start with the Grassmanian $\G(1,4)$, embedded in $\P^9$ under the Plucker embedding, and
cut it with an appropriate $6$-dimensional linear subspace, we obtain
\[E=V(x_0 z-xy,bz-cy,az-cx_0+ax,ay-bx_0+ax_0,bx-cx_0)\subset\P^6_{[x_0\colon a\colon b\colon c\colon x\colon y\colon z]}.\]
Further, if we perform an appropriate linear projection, we obtain the degree-$5$
Kakeya variety  described by the diagram

%
%\[
%\xymatrix{
%E
%\ar@{^{(}->}[rr]\ar[rrdddd]& {}&\P^6_{[x_0:a:b:c:x:y:z]} \ar@{-->}[dddd] & [x_0:a:b:c:x:y:z]\ar@{|->}[dddd]\\
%{}&{}&{}&{}\\
%{}&{}&{}&{}\\
%{}&{}&{}&{}\\
%\P^1_{[t:t_1]}\ar[uuuu]^{ 
%\begin{matrix}
%x_0=t_1\\
%a=\alpha^2 t\\
%b=\alpha t\\
%c=\alpha\gamma t\\
%x=\gamma t_1\\
%y=(1/\alpha-1)t_1\\
%z=\gamma(1/\alpha-1)t_1
%\end{matrix}
%}
%\ar[rr]^{\text{line in direction}}_{[\alpha:1:\gamma]}&{} &\P^3 & [x_0:a-x+y:b-z:c]\\
%}
%\]

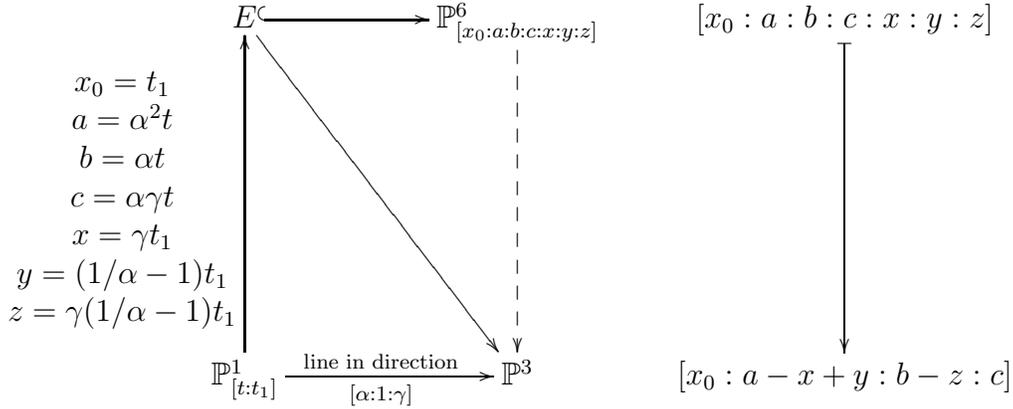
\[
\xymatrix{
E
\ar@{^{(}->}[rr]\ar[rrdddd]& {}&\P^6_{[x_0:a:b:c:x:y:z]} \ar@{-->}[dddd] & [x_0:a:b:c:x:y:z]\ar@{|->}[dddd]\\
{}&{}&{}&{}\\
{}&{}&{}&{}\\
{}&{}&{}&{}\\
\P^1_{[t:t_1]}\ar[uuuu]^{ 
\begin{matrix}
x_0=t_1\\
a=\alpha^2 t\\
b=\alpha t\\
c=\alpha\gamma t\\
x=\gamma t_1\\
y=(1/\alpha-1)t_1\\
z=\gamma(1/\alpha-1)t_1
\end{matrix}
}
\ar[rr]^{\text{line in direction}}_{[\alpha:1:\gamma]}&{} &\P^3 & [x_0:a-x+y:b-z:c]\\
}
\]

where $U=\{[0:\alpha:1:\gamma]\in V(x_0)\ |\  \alpha\neq 0\}.$
This example arises from an investigation in \cite{Slavov_Kakeya_over_alg_closed}.
\label{Grass_example}
\end{Exa}

\subsection{An explicit description}

Now, let $\sigma:U\to F(E,U)$ be a section of the map $F(E,U)\to U$. Shrinking $U$ if necessary, we may assume that $U\subset V(x_0)\cap D_+(x_1)\simeq\A^{n-1}.$ In this case, the composition
$U\xrightarrow{\sigma}F(E,U)\to F(E)\to \text{Lin}^0(\P^1,\P^n)$ actually factors through 
$\text{Lin}^0(\P^1,\P^n)-\text{Lin}^0(\P^1,V(x_1))$. 
There is a map 
$\text{Lin}^0(\P^1,\P^n)-\text{Lin}^0(\P^1,V(x_1))\to V(x_1)$, and hence we obtain a map $U\to V(x_1).$ 
In fact, the map will factor through $V(x_1)\cap D_+(x_0)\simeq\A^{n-1}.$ Regard $U\subset\A^{n-1}$, and let this map $U\to V(x_1)\cap D_+(x_0)$ be given explicitly by 
\begin{align*}
U &\longmapsto\A^{n-1}\\
(u_2,...,u_n) &\longmapsto (\varphi_2(u_2,...,u_n),...,\varphi_n(u_2,...,u_n)).
\end{align*}

Note that if $U$ is properly contained in $V(x_0)\cap D_+(x_1)\simeq \A^{n-1}$, then $\varphi_2,...,\varphi_n$ will be rational functions and may have denominators; for example, if $U= D(g)\subset\A^{n-1}$ is a basic open, then each $\varphi_i\in k_0[x_2,...,x_n]_g$. This happens for instance in the situation of Example \ref{Grass_example}.

Let $K/k_0$ be any field. Then for any $[0:1:u_2:...:u_n]\in U$, the line joining 
$[0:1:u_2:...:u_n]$ and 
$[1:0:\varphi_2(u_2,...,u_n):...:\varphi_n(u_2,...,u_n)]$ is entirely contained in the image of $E(K)\to\P^n(K)$. Note that the intersection of this line with $D_+(x_0)$ is described as
\[\{ (s,su_2+\varphi_2(u_2,...,u_n),...,su_n+\varphi_n(u_2,...,u_n)\ |\ s\in K\}.\]

Say $k_0=\F_{q_0}$ and $K/k_0$ are finite, and we want to prove a lower bound for the size of the image of $E(K)\to\P^n(K)$. Well, instead of the original Kakeya variety $E\to\P^n$, we can consider the map
\begin{align*}
\A^1\times U &\longrightarrow\A^n\\
(s,u_2,...,u_n) &\longmapsto 
(s,su_2+\varphi_2(u_2,...,u_n),...,su_n+\varphi_n(u_2,...,u_n))
\end{align*}
and now we have to give a lower bound for the size of its image on $\F_q$-points.  
Notice, by the way, that for sure, given any $U=D(g)\subset\A^{n-1},$ and given any regular functions
$\varphi_2,...,\varphi_n\in \F_q[x_2,...,x_n]_g$ on $U$, the image on $\F_q$-points of the map above is a Kakeya subset of $\F_q^n$, in the usual combinatorial classical sense (after adding some more $O(q^{n-1})$ points, of course, as usual). Thus, we have reduced the problem of giving a lower bound for the image of $E(\F_q)\to\P^n(\F_q)$ to a very explicit problem.

Focus on the case $U=V(x_0)\cap D_+(x_1).$ Changing notation slightly, now we have $n-1$ polynomials $L_1,...,L_n\in\F_{q_0}[t_1,...,t_{n-1}],$ and we consider the map
\begin{align*}
\varphi:\A^n_{\F_{q_0}}&\longrightarrow\A^n_{\F_{q_0}}\\
(s,t_1,...,t_{n-1})&\longmapsto
(s,st_1+L_1(t_1,...,t_{n-1}),...,st_{n-1}+L_{n-1}(t_1,...,t_{n-1})).
\end{align*}
This is the analogue of the map $I\to E$ from the combinatorial proof of the $2$-dimensional finite field Kakeya problem, 
discussed in the Introduction

The goal is to give a lower bound for the size of the image on $\F_q$-points. Since the case $n=3$ and 
$U=V(x_0)\cap D_+(x_1)$ is already sufficiently interesting and nontrivial, we focus on it in the next sections.

\section{Our approach}

Fix a finite field $\F_{q_0}$ and let $p$ be its characteristic.

\subsection{The main idea}

The main idea of our approach is the Lemma below, based on the Cauchy--Schwarz inequality and the Lang--Weil
estimate. This idea to use the combination of Cauchy--Schwarz and Lang--Weil to give a lower bound for the image set on $\F_q$-points goes back to \cite{U}.

\begin{Lem}
Let $f:X\to Y$ be a morphism of varieties over $F_{q_0},$ where $\dim X=\dim Y=k$ and $X$ is geometrically irreducible. Assume that the fiber product 
$X\times_Y X$ of the morphism $f$ with itself also has dimension $k$. 
Let $C$ be the number of top-dimensional
geometrically irreducible components of $X\times_Y X$. 
For each extension $\F_q/\F_{q_0},$ let $E_{F_q}$ be the image of the induced map $X(\F_q)\to Y(\F_q)$ on $\F_q$-points. Then
\[|E_{\F_q}|\geq\frac{1}{C}q^k-O(q^{k-\frac{1}{2}}),\]
where the implied constant depends only on the complexity of $X$, $Y$, and $f$.    
\label{one_over_C_bound}
\end{Lem}

\begin{Rem}
The important case for us will be when $X$ and $Y$ are fixed. Then the implied constant will depend only on the degree of $f$. 
See Proposition 3.7 in \cite{GW} for an alternative approach when $C=2$ and $f$ is finite and separable. 
\end{Rem}

\begin{proof}
Since
\[
\xymatrix{
(X\times_{Y}X)(\F_q) \ar[r] \ar[d] &X(\F_q)\ar[d]\\
X(\F_q) \ar[r] & E_{\F_q}}\]
is a Cartesian diagram of finite sets, the Cauchy--Schwarz inequality implies
\begin{equation}
\frac{|X(\F_q)|^2}{|E_{\F_q}|}\leq |(X\times_Y X)(\F_q)|
\label{LHS_main_inequality_approach}
\end{equation}
On the other hand, by the Lang--Weil bound (\cite{LW}), we have
\[|X(\F_q)|=q^k+O(q^{k-\frac{1}{2}})\]
(where the implied constant depends only on the complexity of $X$) and
\[|(X\times_Y X)(\F_q)|\leq Cq^k+O(q^{k-\frac{1}{2}})\]
(where the implied constant depends on the complexity of $X$, $Y$, and $f$). The reason for the inequality is that some
of the top-dimensional components of $X\times_Y X$ may not be defined over $\F_{q}$. Combining these, we obtain the desired
conclusion. 
\end{proof}

We note that the two-dimensional variant of Conjecture \ref{main_conj} holds true, and is easy. 

\begin{Prop}
Let $L(t)\in \F_{q_0}[t]$ be an arbitrary polynomial in one variable. Consider the map
\begin{align*}
\A^2_{\F_{q_0}} &\longrightarrow \A^2_{\F_{q_0}}\\
(s,t) &\longmapsto (s,st+L(t)).
\end{align*}
For each extension $\F_q/\F_{q_0}$, let $E_{\F_q}$ be the image of the induced map $\A^2(\F_q)\to\A^2(\F_q)$ on 
$\F_q$-points. Then
\[|E_{\F_q}|\geq\frac{q^2}{2}-O(q^\frac{3}{2}),\]
where the implied constant depends only on the degree of $L$.
\label{easy_2D_Kakeya}
\end{Prop}

\begin{proof}
The fiber product of the given map $\A^2\to\A^2$ with itself can be described explicitly as
\begin{align}
\A^2\times_{\A^2}\A^2 &=\{(s,t_1,t_2)\in\A^3\ |\ st_1+L(t_1)=st_2+L(t_2)\}\notag \\
&=\{(s,t_1,t_2)\in\A^3\ |\ (t-t_1)(s-\widetilde{L}(t_1,t_2))\},
\label{two_diml_fiber_prod_explct}
\end{align}
where $\widetilde{L}$ is defined by $L(t_1)-L(t_2)=(t_1-t_2)\widetilde{L}(t_1,t_2).$
This has two geometrically irreducible components, {\it regardless} of the degree of $L$. 
\end{proof}

\begin{Rem}
In fact, in this $2$-dimensional case, we can remove the error term: $|E_{\F_q}|\geq \frac{q^3}{2q-1}\geq\frac{q^2}{2}.$  The reason is that we can give an explicit count for the number of $\F_q$-points of
(\ref{two_diml_fiber_prod_explct}): there are $q^2$ points where $t=t_1,$ $q^2$ points where $s=\widetilde{L}(t_1,t_2),$ and $q$ points that have been counted twice; total $2q^2-q$. Now the bound without error term follows from (\ref{LHS_main_inequality_approach}).
\label{2d_w_o_error_bound}
\end{Rem}

\begin{Rem}
This estimate, without the error term, is precisely the main result in \cite{Carl}\footnote{Note that this paper states a hypothesis $n<p$ on l. 3 which is never actually used.}. Any Kakeya subset of $\F_q^2$ can be represented as $\{(s,sx+f(x))\ |\ s,x\in\F_q\}$ for some polynomial $f(x)\in\F_q[x]$ by interpolation. So, we can say that \cite{Carl} is exactly the $\frac{q^2}{2}$ bound for Kakeya subsets of $\F_q^2$, and it can be seen as an alternative proof of the $2$-dimensional finite field Kakeya problem, published in 1955 (before the finite field Kakeya problem was even posed). 
\end{Rem}

\begin{Rem}
We can parallel the approach that we present here and the one in \cite{Carl} for the $\frac{q^2}{2}$ bound. 
Namely, equation (2.7) in \cite{Carl} 
modifies readily to higher dimensions to become our
inequality (\ref{LHS_main_inequality_approach}); 
both derivations of this are based on the Cauchy----Schwarz inequality (it is just that our approach is slightly more direct, as we use Cauchy----Schwarz once while Carlitz uses it twice). Also, Carlitz's  equation (2.8) obtained by an elementary exponential sums argument is exactly our count for the number of $\F_q$-points in
(\ref{two_diml_fiber_prod_explct}) of Remark \ref{2d_w_o_error_bound}. One way or another, the reason the $2$-dimensional case is easy is that we can give an explicit count for the number of $\F_q$-points in the fiber product (\ref{two_diml_fiber_prod_explct}); in higher dimensions, we will need to use the Lang--Weil bound. 
\end{Rem}

\subsection{Indecomposability of certain polynomials}

We will give two proofs of Proposition \ref{separated_vars},
 both of which make substantial use of the case $e=0$ in the Lemma below. The case $e=2$ will be used later in Section \ref{subsection_case_mixed_vars} in the proof of Proposition \ref{mixed_vars_prop}.  

\begin{Lem} Let $e\in\{0,2\}.$ When $e=2$, assume for convenience that $p>2$. 
Let $f(x)\in \overline{\F_p}[x]$ be a polynomial. Suppose that there exist polynomials 
$Q(t)\in \overline{\F_p}[t]$ and $\lambda(x,y)\in \overline{\F_p}[x,y]$ with
$\deg Q\geq 2$ such that
\[(x-y)^e\frac{f(x)-f(y)}{x-y}=Q(\lambda(x,y))\]
as polynomials in $\overline{\F_p}[x,y].$ Then $f(x)$ is a linearized polynomial. 
\label{truly_poly_in_two_vars}
\end{Lem}

\begin{proof}
Throughout the proof, we will be using the following fact: if $t=p^cN$ with $p\nmid N,$ then $x=1$ is a root of the polynomial $x^{t-1}+x^{t-2}+\dots+x+1$ of multiplicity exactly $p^c-1.$ This is so because 
\[\frac{x^t-1}{x-1}=\frac{(x^N-1)^{p^c}}{x-1},\]
and $x=1$ is a simple root of $x^N-1$. Equivalently, in the factorization of $x^{t-1}+x^{t-2}y+\dots+y^{t-1}\in \overline{\F_p}[x,y],$
the multiplicity of the linear factor $x-y$ is exactly $p^c-1.$ Also, when $N$ is not a power of $p$, the polynomial $x^N-1$ has a root other than $x=1$. 

Let $d=\deg f$, $m=\deg Q\geq 2,$ and $s=\deg\lambda,$ so $e+d-1=ms.$ Write $f(x)=\sum_{t=0}^d a_t x^t$. 
Write $\lambda=\lambda_s+\lambda_{s-1}+\dots+\lambda_0,$ where each $\lambda_i$ is homogeneous of degree $i$.  
By assumption,
\begin{equation}
(x-y)^e\frac{f(x)-f(y)}{x-y}=b_0(\lambda_s+\lambda_{s-1}+\dots+\lambda_1+\lambda_0)^m+b_1(\lambda_s+\dots+\lambda_0)^{m-1}+\dots,\label{homog_equate}
\end{equation}
where $b_0\neq 0.$
Comparing the top homogeneous parts above and setting $y=1$, we deduce that 
\[a_d(x-1)^e(x^{d-1}+x^{d-2}+\dots+x+1)=b_0\lambda_s(x,1)^m.\]
Write $d=p^aN$ with $p\nmid N$ and $a\geq 0$. 
If $N>1$ and $\zeta\neq 1$ is an $N$-th root of $1$ in $\overline{\F_p}$, then $x-\zeta$ appears on the LHS with multiplicity $p^a$, hence $m|p^a|d,$ which is impossible, since $m|e+d-1=d\pm 1$. Therefore, $d=p^a$, and so, up to a nonzero factor, $\lambda_s=(x-y)^s$.

Note that $s<p^a-p^{a-1}$ unless $e=2,p=3,a=1$. Indeed, if $s\geq p^a-p^{a-1},$ after multiplying both sides by $m\geq 2,$ we would obtain $e+p^a-1=sm\geq 2(p^a-p^{a-1})$. When $e=0,$ this is clearly impossible. When $e=2,$ we are assuming $p>2$, so this inequality is again impossible, unless $p=3,a=1$. We postpone 
this case and handle it separately. 

We claim that $\lambda_k=0$ for each $k\in\{1,...,s-1\}.$ We argue by descending induction on $k$. 
Fix $k\in\{1,...,s-1\}$ and suppose that for all $k'$ with $k<k'<s,$ we have $\lambda_{k'}=0.$
Consider the homogeneous components on both sides of (\ref{homog_equate}) of degree $sm-s+k.$
The induction hypothesis implies that $\lambda_s^{m-1}\lambda_k$ is the only term that contributes to the RHS
(note also that $sm-s+k>s(m-1)$), and hence,
letting $t=sm-s+k+1-e,$ we obtain
\[a_t(x-y)^e(x^{t-1}+...+y^{t-1})=b_0m\lambda_s^{m-1}\lambda_k.\]
Note that $p\nmid m$, as $p^a\pm 1=sm.$

Write $t=p^cN$ with $p\nmid N.$ Suppose that $a_t\neq 0.$ Comparing the multiplicity of the factor $x-y$ on the LHS and RHS above, we obtain $e+p^c-1\geq sm-s.$ But, $t<p^a$ and so $c\leq a-1,$ giving the chain of inequalities
\[e+p^{a-1}-1\geq e+p^c-1\geq sm-s=e+p^a-1-s.\]
However, this contradicts the inequality $s<p^a-p^{a-1}$ that we obtained earlier. 
Therefore, $a_t=0$ and $\lambda_k=0.$
This completes the induction step. 

Suppose that the coefficient $a_t$ of $x^t$ in $f(x)$ is nonzero. Comparing the homogeneous terms of degree $t-1+e$ in
(\ref{homog_equate}), we deduce that
\[a_t(x-y)^e(x^{t-1}+\dots+y^{t-1})=c_t\lambda_s^l\]
for some constant $c_t$ and some integer $l$. If $t$ is not a power of $p$, the LHS would have a linear 
factor besides $x-y$, while the RHS is a power of $x-y.$ 

We are left with the case $e=2,p=3,d=3.$ Without loss of generality, $f$ is monic. Say
\[(x-y)^2(x^2+xy+y^2+a_2(x+y)+a_1)=(\lambda_2+\lambda_1+\lambda_0)^2+b_1(\lambda_2+\lambda_1+\lambda_0)+b_2.\]
Compare the degree-$3$ homogeneous parts on both sides:
\[a_2(x-y)^2(x+y)=2\lambda_2\lambda_1.\]
So, $\lambda_1$ is a multiple of $x+y$.
Compare now the homogeneous terms of degree $2$:
\[a_1(x-y)^2=(2\lambda_0+b_1)\lambda_2+\lambda_1^2.\]
This implies that $(x-y)|\lambda_1,$ and so $\lambda_1=0.$ The proof finishes as in the main case, considered above. 
\end{proof}

\begin{Def}
For a polynomial $f(x)\in\F_{q_0}[x],$ define $\widetilde{f}(x,y)\in\F_{q_0}[x,y]$ via 
\[f(x)-f(y)=(x-y)\widetilde{f}(x,y).\]
\end{Def}

\section{Main results}

\subsection{The case of separated variables}

Fix a finite field $\F_{q_0}$ and let $p$ be its characteristic.
We now give two proofs of Proposition \ref{separated_vars}. 

Linearized polynomials, after perturbations
by linear terms, have large image sets on $\F_q$-points. 

\begin{Lem}
Let $f(x)\in\F_{q}[x]$ be a linearized polynomial with coefficients in a finite field $\F_q$. Assume that the characteristic $p$ of $\F_q$ is odd. Then for at least $\frac{p-2}{p-1}q$ values of $a\in\F_q,$ the polynomial $f(x)+ax$ is a permutation polynomial of $\F_q$. 
\label{many_perturbations_of_lin_are_PP}
\end{Lem}

\begin{proof}
This follows from the Remarks succeeding Theorem 1 and Conjecture 2 in \cite{EGN}. We include the argument here. 
Since $f$ is linearized, for each $a\in\F_q$, we have that $f(x)+ax$ is an $\F_p$-linear map $\F_q\to\F_q.$  
If it is not a permutation polynomial, it will have a kernel of dimension at least one, hence size at least $p$.
Thus, in this case, there will be at least $p-1$ values of $x\in\F_q^*$ which map to $a$ under the map
$F_q^*\to \F_q, x\mapsto -\frac{f(x)}{x}.$ So, the number of values of $a$ such that $f(x)+ax$ is not a permutation polynomial is at most $\frac{q-1}{p-1}.$
\end{proof}

We are now ready to give the first proof of Proposition \ref{separated_vars}. In the case when 
both $L$ and $M$ are linearized, we assume that $p\geq 5.$

\begin{proof}[First proof of Proposition \ref{separated_vars}]
Suppose first that at least one of $L(t_1),M(t_2)$ is not a linearized polynomial. Then
at least one of $\widetilde{L}(t_1,t_1'),\widetilde{M}(t_2,t_2')$ is not decomposable, by Lemma
\ref{truly_poly_in_two_vars}. Therefore, by a theorem of Schinzel (see \cite{Schinzel}), the polynomial
$\widetilde{L}(t_1,t_1')-\widetilde{M}(t_2,t_2')$ is irreducible. Take the fiber product of the given map $\varphi:\A^3\to\A^3$ with itself; this fiber product is explicitly given by
\[V\left(  (t_1-t_1')(s-\widetilde{L}(t_1,t_1')),(t_2-t_2')(s-\widetilde{M}(t_2,t_2'))\right)\subset\A^5_{s,t_1,t_1',t_2,t_2'}.\]
Therefore, it has $4$ irreducible components of top dimension, namely:
$V(t_1-t_1',t_2-t_2'), V(t_1-t_1',s-\widetilde{M}(t_2,t_2')),
V(t_2-t_2',s-\widetilde{L}(t_1,t_1')), V(s-\widetilde{M}(t_2,t_2'),s-\widetilde{L}(t_1,t_1')).$
Note that 
$V(s-\widetilde{M}(t_2,t_2'),s-\widetilde{L}(t_1,t_1'))\simeq V(\widetilde{L}(t_1,t_1')-\widetilde{M}(t_2,t_2'))\subset\A^4$ is indeed irreducible, by the result of Schinzel. 
So, in this case, the conclusion follows by Lemma \ref{one_over_C_bound}. 

Suppose now that both $L$ and $M$ are linearized, and $p\geq 5.$ There are at most $\frac{q}{4}$ values of $s\in\F_q$ such that 
$L_s(t):=L(t)+st$ is not a permutation polynomial; similarly, there are 
at most $\frac{q}{4}$ values of $s\in\F_q$ such that $M_s(t):=M(t)+st$ is not a permutation polynomial. Overall, there are at least
$\frac{q}{2}$ values of $s\in\F_q$ such that both $L_s$ and $M_s$ are permutation polynomials. Thus, the total image set has size at least $\frac{q}{2}.q.q$ in this case (without error term).
In fact, this bound can be improved in larger characteristic. 
\end{proof}

The second proof of Proposition \ref{separated_vars} that we give is based on the Lemma below,
in place of Schinzel's irreducibility theorem. 

\begin{Lem}
Let $L(x)\in \overline{\F_p}[x]$ be any polynomial which is not linearized. 
For $a\in \overline{\F_p},$ define $L_a(x)=L(x)+ax.$ Then 
\[|\{a\in \overline{\F_p}\ |\ \widetilde{L_a}(x,y)\ \text{is reducible}\}|<\deg L.\]
\end{Lem}

\begin{proof}
By Lemma \ref{truly_poly_in_two_vars}, we know that $\widetilde{L}(x,y)$ is not of the form $Q(\lambda(x,y)),$ 
where $\deg Q>1$. Now, by Corollary 1 in \cite{L}, for all but at most $\deg L-1$ values of $a$, the polynomial 
$\widetilde{L_a}(x,y)=\widetilde{L}(x,y)+a$
will be irreducible.
\end{proof}

In the second proof of Proposition \ref{separated_vars}, we assume that $p\geq 3$ when exactly one of $L,M$ is linearized, and $p\geq 5$ when both $L,M$ are linearized. 

\begin{proof}[Second proof of Proposition \ref{separated_vars}]
Suppose first that none of $L$ and $M$ is linearized. For at least $q-(\deg(L)+\deg(M))$ values of 
$s\in\F_q,$ both polynomials $\widetilde{L_a}(x,y)$ and $\widetilde{M_a}(x,y)$ are geometrically irreducible,
hence the image sets of $L_s(t_1)$ and  $M_s(t_2)$ each have size at least $\frac{q}{2}-O(\sqrt{q}).$ Overall, the size of the image set 
$\varphi(\F_q^3)$
is then at least $q\frac{q}{2}\frac{q}{2}-O(q^{\frac{5}{2}}).$

Suppose that $L$ is linearized but $M$ is not, and $p\geq 3.$ For at least $\frac{q}{2}$ values of $s\in\F_q,$ $L_s$ is a 
permutation polynomial of $\F_q.$ Also, for at least $q-\deg(M)$ values of $s\in\F_q,$ the polynomial 
$\widetilde{M_s}(x,y)$ is geometrically irreducible. Overall, for $\frac{q}{2}$ values of $s\in\F_q,$ we know that $L_s$ is a permutation polynomial and $\widetilde{M_s}(x,y)$ is geometrically irreducible, hence $M_s$ has image set of size at least $\frac{q}{2}-O(\sqrt{q})$. Therefore, the total image size is at least 
$\frac{q}{2}.q.\frac{q}{2}-O(q^{\frac{5}{2}}).$

When both $L,M$ are linearized, we finish as in the first proof. 
\end{proof}

\subsection{The case of mixed variables}
\label{subsection_case_mixed_vars}

In this section, we prove Proposition \ref{mixed_vars_prop}.

\begin{Lem}
Let $k$ be any algebraically closed field. 
Let $\widetilde{f}(t_2,t_2'),\widetilde{g}(t_1,t_1')$ be two polynomials, not both zero, and such that
$(t_2-t_2')^2\widetilde{f}(t_2,t_2')-(t_1-t_1')^2\widetilde{g}(t_1,t_1')\in k[t_1,t_1',t_2,t_2']$ has at most $t$
irreducible factors. Consider the variety
\[X=V(s(t_1-t_1')+(t_2-t_2')\widetilde{f}(t_2,t_2'), s(t_2-t_2')+(t_1-t_1')\widetilde{g}(t_1,t_1'))\subset\A^5_{s,t_1,t_1',t_2,t_2'}.\]
Then $\dim X=3,$ and $X$ has at most $t+1$ irreducible components of maximal dimension. 
\label{most_basic_lem_mixed_vars}
\end{Lem}

\begin{proof}
Let $Z$ be an irreducible component of $X$ of top dimension; we know that $\dim Z\geq 3$. Set $\text{Diag}=V(t_1-t_1',t_2-t_2').$ Note also that both
$\widetilde{f}$ and $\widetilde{g}$ have to be nonzero. 

Suppose first that $Z\subset V(t_1-t_1').$ Then $Z\subset V(s(t_2-t_2'))$ and so 
$Z\subset V(t_1-t_1',s)\cup V(t_1-t_1',t_2-t_2').$ Since these are irreducible and $3$-dimensional, either $Z=V(t_1-t_1',s)$, or $Z=\text{Diag}.$ The former case is impossible: take any $t_2,t_2'$ with $t_2\neq t_2', \widetilde{f}(t_2,t_2')\neq 0$;
then the point $(0,0,0,t_2,t_2')$ belongs to $Z$ but not to $X$. 
So, $Z\subset V(t_1-t_1')$ implies $Z=\text{Diag}$. 
Similarly, $Z\subset V(t_2-t_2')$ implies $Z=\text{Diag}$. 

Assume from now on that a generic point in $Z$ satisfies
$t_1\neq t_1', t_2\neq t_2'$, i.e., $Z\cap\{t_1\neq t_1',t_2\neq t_2'\}$ is an open dense subset of $Z$.  

Let
\[T=V((t_2-t_2')^2\widetilde{f}(t_2,t_2')-(t_1-t_1')^2\widetilde{g}(t_1,t_1'))\subset\A^4_{t_1,t_1',t_2,t_2'}.\]
By assumption, $T$ has at most $t$ irreducible components, each of them of dimension $3$. Since 
$\hat{T}:=T\cap\{t_1\neq t_1',t_2\neq t_2'\}$ is open in $T$, it has at most $t$ irreducible components, each of them of dimension $3$.

Note that the map
\begin{align*}
X\cap\{t_1\neq t_1',t_2\neq t_2'\} &\longrightarrow T\cap\{t_1\neq t_1', t_2\neq t_2'\}\\
(s,t_1,t_1',t_2,t_2') &\longmapsto (t_1,t_1',t_2,t_2')
\end{align*}
is an isomorphism,  with inverse 
\[(t_1,t_1',t_2,t_2')\mapsto \left(-\frac{(t_2-t_2')\widetilde{f}(t_2,t_2')}{t_1-t_1'},t_1,t_1',t_2,t_2'\right).\]

Consider the diagram
\[
\xymatrix{
 Z \ar@{^{(}->}[r]^-{\text{closed}} & X &{}\\
 Z\cap\{t_1\neq t_1',t_2\neq t_2'\}\ar@{^{(}->}[r]_-{\text{closed}}\ar@{^{(}->}[u]_-{\text{open dense}} & 
 X\cap\{t_1\neq t_1',t_2\neq t_2'\} \ar@{^{(}->}[u]_-{\text{open}}\ar[r]^-{\simeq} & T\cap\{t_1\neq t_1',t_2\neq t_2'\}
}
\]
Note that
\[\dim Z=\dim (Z\cap\{t_1\neq t_1',t_2\neq t_2'\})\leq 3=\dim (X\cap\{t_1\neq t_1',t_2\neq t_2'\})\leq \dim X,\]
and hence the assumption $\dim Z=\dim X$ implies that 
this common dimension has to equal 
$3$.  The first horizontal arrow on the bottom is a closed embedding between varieties of the same dimension, 
and since $Z\cap\{t_1\neq t_1',t_2\neq t_2'\}$ is irreducible,
it has to be one of the irreducible components of $X\cap\{t_1\neq t_1',t_2\neq t_2'\}$. The latter is isomorphic to $\hat{T}$ and thus
has at most $t$ components. Therefore, $Z$ is the Zariski closure in $X$ of one of the components of 
$X\cap\{t_1\neq t_1',t_2\neq t_2'\},$ hence there are at most $t$
possibilities for $Z$. Counting in $\text{Diag},$ we deduce that indeed, $X$ has at most $t+1$ top--dimensional irreducible components. 
\end{proof}

We will need the following easy preparation:
\begin{Lem}
For polynomials $L,M$ in one variable, the number of factors of $xL(x)-yM(y)\in k[x,y]$ equals the number of factors of 
$(t_2-t_2')L(t_2-t_2')-(t_1-t_1')M(t_1-t_1')\in k[t_1,t_1',t_2,t_2'].$
\end{Lem}

\begin{proof}
The map
\begin{align*}
V((t_2-t_2')L(t_2-t_2')-(t_1-t_1')M(t_1-t_1')) &\longrightarrow V(xL(x)-yM(y))\times\A^2\\
(t_1,t_1',t_2,t_2')   &\longmapsto (t_2-t_2',t_1-t_1',t_2',t_1')
\end{align*}
is an isomorphism, with inverse $(x,y,p,q)\mapsto (y+q,q,x+p,p)$, and hence these two varieties have the same number of
irreducible components. 
\end{proof}

The new ingredient that we will need is the following result of M. Zieve \cite{Z}: 

\begin{thm} Let $p>2$. Suppose that $f,g$ are linearized polynomials over $\overline{\F_p}$ with $f'(0)g'(0)\neq 0$ and $f(0)=0, g(0)=0$. Then 
$xf(x)-yg(y)$ has at most $3$ irreducible factors. 
\end{thm}

\begin{proof}[Proof of Proposition \ref{mixed_vars_prop}]

The fiber product of the map $\varphi$ with itself is  the variety
\[X=V(s(t_1-t_1')+(t_2-t_2')\widetilde{L}(t_2,t_2'), s(t_2-t_2')+(t_1-t_1')\widetilde{M}(t_1,t_1'))\subset\A^5_{s,t_1,t_1',t_2,t_2'}.\]
If it is $3$-dimensional and has only $2$ components of top dimension, then the size of the image on $\F_q$-points of $\varphi$ will be at least $\frac{q^3}{2}-O(q^\frac{5}{2}).$ So, we have to consider the case when the polynomial $(t_2-t_2')^2\widetilde{L}(t_2,t_2')-(t_1-t_1')^2\widetilde{M}(t_1,t_1')\in \overline{\F_p}[t_1,t_1',t_2,t_2']$ is reducible.
By Schinzel's theorem and the case $e=2$ of Lemma \ref{truly_poly_in_two_vars}, this can happen only when both $L$ and $M$ are linearized. We can assume that $L(0)=M(0)=0,$ since a shift does not change the size of $\varphi(\F_q^3)$. 
 
So, let $L,M$ be linearized polynomials with $L'(0)M'(0)\neq 0$ and $L(0)=0, M(0)=0.$ 
The number of factors 
in $\overline{\F_p}[t_1,t_1',t_2,t_2']$ 
of $(t_2-t_2')^2\widetilde{L}(t_2,t_2')-(t_1-t_1')^2\widetilde{M}(t_1,t_1')=(t_2-t_2')L(t_2-t_2')-(t_1-t_1')M(t_1-t_1')$ equals the number of factors 
in $\overline{\F_p}[x,y]$ of
$xL(x)-yM(y),$ which is at most $3$, by Zieve's theorem. So, the statement follows from Lemma \ref{most_basic_lem_mixed_vars} with $t=3.$
\end{proof}

We finish with two more special cases of Conjecture
\ref{main_conj} in the case of mixed variables. 
There is one obvious case when the fiber product $X$
of $\varphi$ with itself can acquire many components, namely, when $L=M$. We handle this case now. 

\begin{Lem}
Let $f(t)\in\overline{\F_p}[t]$ be any linearized polynomial. Assume that $p\geq 5.$ Let $L(t_1,t_2)=f(t_2), M(t_1,t_2)=f(t_1).$ Then,
notation as in Conjecture \ref{main_conj}, we have:
\[|E_{\F_q}|\geq \frac{p-3}{p-1}q^3\geq \frac{q^3}{2}.\]
\end{Lem}

\begin{proof}
Without loss of generality, $f(0)=0.$

Let $B=\{s\in\F_q\ |\ f(x)+sx\ \text{is not a permutation polynomial over
$\F_q$}\}$; then we know from Lemma \ref{many_perturbations_of_lin_are_PP} that $|B|\leq\frac{q}{p-1}$ (neglecting the $O(1)$ term). So, $|B\cup (-B)|\leq\frac{2q}{p-1}.$ 
Let $B'=(B\cup (-B))^c,$ so for any $s\in B',$ both $f(x)\pm sx$ are permutation polynomials, and $|B'|\geq \frac{p-3}{p-1}q$. 

We claim that $B'\times\F_q\times\F_q\subset E_{\F_q}.$ Fix any $(s,\beta,\gamma)\in B'\times\F_q\times\F_q.$ Let $x\in\F_q$ be such that $f(x)-sx=\gamma-\beta$, and let $t_2\in\F_q$ be such that $f(t_2)+st_2=\beta-sx.$
Let $t_1=t_2+x$. Then $(s,t_1,t_2)$ maps to $(s,\beta,\gamma).$
\end{proof}

\begin{Rem}
This Lemma gives examples of maps $\A^3_{\F_p}\to\A^3_{\F_p}$ with large image on $\F_q$-points, which are not bijective.  Contrast with the $\frac{5}{6}$ bound of Theorem 1.2 in \cite{GW}.
\end{Rem}

One final special case is handled in the following

\begin{Lem}
Suppose that $L(t_1,t_2)=L(t_2)$ depends only on the second variable, $M(t_1,t_2)=M(t_1)$ depends only on the first variable, and $\deg_{t_1}M\leq 1.$ Then, notation as in Conjecture \ref{main_conj}, for any $\F_q/\F_{q_0}$, we have
\[|E_{\F_q}|\geq\frac{q^3}{3}-O(q^{\frac{5}{2}}).\]
\end{Lem}

\begin{proof}
Without loss of generality, $M(0)=0$ (replacing $M$ by $M-M(0)$ only shifts the last coordinates of the value sets, leaving the size unchanged). Write $M(t_1)=at_1, a\in\F_{q_0}$. The case $a=0$ is easy: we are dealing with the map $(s,t_1,t_2)\longmapsto (s,st_1+L(t_2),st_2).$ For any $(\alpha,\beta,\gamma)\in\F_q^3$ with $\alpha\neq 0,$ take $s=\alpha, t_2=\frac{\gamma}{\alpha},$ and solve $st_1+L(t_2)=\beta$ for $t_1.$ In this case, the size of the image of the map is at least $q^3-q^2$. Assume from now on that $a\neq 0,$
so we are considering the map
\[\F_q^3\to\F_q^3,\ (s,t_1,t_2)\mapsto (s,st_1+L(t_2),st_2+at_1).\] 

Fix $\gamma\in\F_q.$ We will count the number of points in the image of the above map with last coordinate $\gamma$, and show that their number is at least $\frac{q^2}{3}-O(q^\frac{3}{2}).$

The condition that the last coordinate is $\gamma$ is $t_1=\frac{\gamma-st_2}{a}.$ Now setting $t=t_2$, we are looking at the map
\[\A^2\to\A^2, \ 
(s,t)\mapsto \left(s,\frac{s\gamma}{a}-\frac{s^2 t}{a}+L(t)\right).\]
The fiber product of this map with itself is given by
\[
\{(s,t,t')\in\A^3\ |\ \frac{1}{a}(t-t')(s^2-a\widetilde{L}(t,t'))=0\}.\]
This has either $2$ or $3$ irreducible components of top dimension, depending on whether $\widetilde{L}(t,t')$ is a square in $\overline{\F_p}[t,t'].$ The conclusion now follows from Lemma \ref{one_over_C_bound}. 
\end{proof}

\subsection{Open questions}

Unfortunately, if we take the fiber product of the map $\varphi$ in Conjecture \ref{main_conj} with itself, we cannot characterize the cases when we get more than $4$ geometrically irreducible components. 
Explicitly, this fiber product is given by the two equations
\begin{equation}
\begin{aligned}
&s(t_1-t_1')+L(t_1,t_2)-L(t_1',t_2')=0\\
&s(t_2-t_2')+M(t_1,t_2)-M(t_1',t_2')=0
\label{pencil_surfaces}
\end{aligned}
\end{equation}
in $\A^5_{s,t_1,t_2,t_1',t_2'}$, and it is not clear how to control the number of irreducible
components of top dimension. If one carefully modifies the argument in Lemma \ref{most_basic_lem_mixed_vars},
this investigation would reduce to the following

\begin{Que}
Is it possible to characterize the cases when a polynomial
\[(t_2-t_2')\left(L(t_1,t_2)-L(t_1',t_2')\right)-(t_1-t_1')\left(M(t_1,t_2)-M(t_1',t_2')\right)\]
in $\overline{\F_p}[t_1,t_2,t_1',t_2']$ is reducible? Or, thinking of
(\ref{pencil_surfaces}) as a pencil of surfaces in $\A^4$ with parameter $s$, it is true that for all but $O_{\deg(L),\deg(M)}(1)$ values of $s$,  
the corresponding surface has at most $4$ irreducible components of dimension $2$, except in certain cases that we can classify? Or, is it true that for at least $\frac{q}{2}$ values of $s\in\F_q$, the corresponding surface is geometrically irreducible, again except in a certain list of cases?
\label{question_reducibility}
\end{Que}

The reason we hope that our special cases give sufficient evidence for Conjecture \ref{main_conj}
is that polynomials of fewer variables in lower--dimensional affine spaces are more likely to be reducible, so in fact, we think that the cases we have handled are the ``worst" cases, as long as
our conjecture is concerned. 

\section*{Acknowledgments}
This research was performed while the author was visiting the Institute for Pure and Applied Mathematics (IPAM), which is supported by the National Science Foundation.
I thank Terry Tao for the extremely fruitful, inspiring, and encouraging discussions during my IPAM participation. 
I am gratefully indebted to Michael Zieve for the numerous discussions, suggestions, and references, and specifically for proving the result of \cite{Z} that I asked him about. 
 I also thank 
Kiran Kedlaya for some discussions, and  
 Greta Panova for a suggestion concerning Lemma \ref{truly_poly_in_two_vars}.

\end{document}